\documentclass{daj}

\usepackage{amsmath}
\usepackage{amsfonts}
\usepackage{amsthm}

\theoremstyle{plain}
\newtheorem{theorem}{Theorem}[section]
\newtheorem{lemma}[theorem]{Lemma}
\newtheorem{corollary}[theorem]{Corollary}

\newtheorem{proposition}[theorem]{Proposition}

\newtheorem*{theorem*}{Theorem}
\newtheorem{problem}{Problem}
\newtheorem*{problem*}{Problem}

\theoremstyle{definition}
\newtheorem{definition}[theorem]{Definition}

\newtheorem*{claim*}{Claim}
\newtheorem*{remark}{Remark}
\newtheorem*{remarks}{Remarks}

\newcommand{\lE}{\mathbb{E}^{\log}}
\newcommand{\C}{{\mathbb C}}
\newcommand{\E}{{\mathbb E}}

\newcommand{\N}{{\mathbb N}}
\renewcommand{\P}{{\mathbb P}}
\newcommand{\Q}{{\mathbb Q}}
\newcommand{\R}{{\mathbb R}}
\renewcommand{\S}{\mathbb{S}}
\newcommand{\T}{{\mathbb T}}

\newcommand{\CM}{{\mathcal M}}

\newcommand{\CX}{{\mathcal X}}
\newcommand{\Z}{{\mathbb Z}}

\newcommand{\inv}{^{-1}}
\newcommand{\bN}{{\mathbf{N}}}

\newcommand{\norm}[1]{\left\Vert #1\right\Vert}

\DeclareMathOperator{\id}{id}

\dajAUTHORdetails{%
	title = {Good Weights for the Erd\H{o}s Discrepancy Problem}, 
	author = {Nikos Frantzikinakis},
	plaintextauthor = {Nikos Frantzikinakis},
	plaintexttitle = {Good Weights for the Erdos Discrepancy Problem}, 
	keywords ={Multiplicative functions, discrepancy,  Erd\H{o}s discrepancy problem, Elliott conjecture, Furstenberg correspondence.},
}   

\dajEDITORdetails{%
	year={2020},
	number={8},
	received={30 April 2020},   
	published={14 July 2020},  
	doi={10.19086/da.13688},       
}   

\begin{document}
	
	\begin{frontmatter}[classification=text]
		
		\title{Good Weights for the Erd\H{o}s Discrepancy Problem} 
		

		\author[Nikos]{Nikos Frantzikinakis\thanks{Supported by the
					 Hellenic Foundation for	Research and Innovation,	Project 	No: 1684.}}


\begin{abstract}
The Erd\H{o}s discrepancy problem, now a theorem by T.~Tao, asks whether
every sequence with values plus or minus one has unbounded discrepancy along all homogeneous arithmetic progressions. We establish weighted variants of this problem,
  for weights given  either by structured sequences that enjoy some irrationality features, or  certain random sequences. As an intermediate result, we establish that weighted sums of
   bounded multiplicative functions and products of shifts of such functions are unbounded. A key ingredient in our analysis for the structured weights, is a structural result for measure preserving systems
naturally associated with bounded  multiplicative functions  that was recently obtained in joint work with B.~Host.
\end{abstract}

\end{frontmatter}

\maketitle

\section{Introduction and main results}\label{S:MainResults}
\subsection{Introduction}The Erd\H{o}s discrepancy problem is an elementary question that dates back to the 1930's and asks if  there is a sequence  $a\colon \N\to \{-1, 1\}$ that is evenly distributed along all homogeneous arithmetic progressions, in the sense that the sequence of partial sums $(\sum_{k=1}^n a(dk))_{n\in\N}$ is bounded  uniformly in $d\in \N$. The problem remained dormant for a long time and it was not until 2010 that interest was rejuvenated,
when it became the subject of the  \texttt{Polymath5} project  (see \cite{G13, P5} for related details).
The problem was finally solved in 2015  by  T.~Tao  \cite{T16}
who proved the following
(henceforth, with $\S$ we denote the unit circle and with $\mathbb{U}$ the complex unit disc):
\begin{theorem}[Tao \cite{T16}]\label{T:Tao}
	For every sequence  $a\colon \N\to \S$ we have
	\begin{equation}\label{E:EDP}
	\sup_{d, n\in \N}\Big|\sum_{k=1}^n a(dk) \Big|=+\infty.
	\end{equation}
\end{theorem}
We seek to obtain weighted variants of the previous result. To facilitate exposition, we introduce the following notion:

\begin{definition}
	We say that a  sequence $w\colon \N\to \mathbb{U}$ is {\em a  good weight for the Erd\H{o}s discrepancy problem}, or simply, {\em a good weight}, if for
	every $a\colon \N\to \S$ we have
	\begin{equation}\label{E:WEDP}
		\sup_{d, n\in \N}\Big|\sum_{k=1}^n a(dk) \, w(k)\Big|=+\infty.
	\end{equation}
\end{definition}
Theorem~\ref{T:Tao}  implies that  $w=1$ (and more generally  $w=f$ where $f\colon \N\to \S$ is a completely multiplicative function)  is a good weight for the Erd\H{o}s discrepancy problem.  On the other hand, sequences with bounded partial sums, like
the sequence $(e(k\alpha))_{k\in\N}$, where $\alpha\in \R\setminus\Z$ and $e(t):=e^{2\pi i t}$, are not  good weights, and  more generally,  a product of a completely multiplicative function $f\colon \N\to\S$
with a sequence that has bounded partial sums is not a good weight (take  $a=\bar{f}$).
It is less clear  if some other    oscillatory sequences  like $(e(k^l\alpha))_{k\in\N}$, where $l\geq 2$ and  $\alpha$ is irrational, or random sequences of $\pm 1$'s
are good weights.
We will show  in Corollary~\ref{C:EDP1} and Theorem~\ref{T:Random1} that they are; that is,  for every $a\colon \N\to \S$ we have
$$
 	\sup_{d, n\in \N}\Big|\sum_{k=1}^n a(dk) \,e(k^l\alpha)\Big|=+\infty
$$
and a similar statement holds if we use as weights  random sequences of $\pm 1$.
Moreover, in Theorem~\ref{T:EDP1} we give a rather   general criterion that allows us to show that a  large class of  zero entropy sequences  that enjoy  certain irrationality features are good weights for the Erd\H{o}s discrepancy problem.

On a related result of independent interest, we show that certain weighted sums
of multiplicative functions are unbounded. For instance,
 we prove in Corollary~\ref{C:mf1} that if $l\geq 2$,  $\alpha$ is irrational, and $f, g \colon \N\to \S$ are multiplicative functions, then
	$$
	\sup_{ n\in \N}\Big|\sum_{k=1}^n f(k)\,  g(k+1) \, e(k^l\alpha)\Big|=+\infty,
	$$
and  in Theorems~\ref{T:Random2}  we  prove an analogous  result when the  weights are given by  random sequences of $\pm 1$'s.

\subsection{Results related to the weighted Erd\H{o}s discrepancy problem}
 The next result gives   necessary conditions for a bounded sequence of complex numbers to be a good weight for the  Erd\H{o}s discrepancy problem. In order to explain the exact assumptions needed, we use ergodic terminology that is
explained in Section~\ref{SS:ergodicnotation}, and in Corollary~\ref{C:EDP1} we give some  explicit examples. See also Section~\ref{SS:def} for our notation regarding averages; for reasons that are explained in Section~\ref{SS:ergodicnotation} we use logarithmic averages.

\begin{definition}
We say that the sequence   $a\colon \N\to \mathbb{U}$
	\begin{itemize}
\item	has  {\em vanishing self-correlations}, if
	for every $h\in\N$ we have
	$$
	\lE_{n\in\N}\, a(n+h)\, \overline{a(n)}=0;
	$$
\item is  {\em non-null for logarithmic averages}, or simply, {\em non-null},  if
	$$
	\liminf_{N\to\infty}\lE_{n\in[N]}\, |a(n)|^2>0.
	$$
	\end{itemize}
\end{definition}
Our main result regarding structured (zero entropy) weights  is the following one:
\begin{theorem}\label{T:EDP1}
	Suppose that   $w\colon \N\to \mathbb{U}$     is non-null and  totally ergodic,  and has zero entropy and  vanishing self-correlations. Then $w$ is a good weight for the Erd\H{o}s discrepancy problem.
\end{theorem}

\begin{remarks}

\hfill \begin{itemize}
\item As was the case in \cite{T16}, if $\mathcal{H}$ is an arbitrary  inner product space and  $a\colon \N\to \mathcal{H}$ is such that $\norm{a(k)}_{\mathcal{H}}=1$ for all $k\in \N$,  then our argument works without any change and shows that $$\sup_{d, n\in \N}\norm{\sum_{k=1}^n w(k) \, a(dk)}_{\mathcal{H}}=+\infty.$$
\item Using Theorem~\ref{T:mf1} below, it is straightforward to adapt the proof of  Theorem~\ref{T:EDP1} in order to get the following stronger conclusion: For $Q(k)=\prod_{j=1}^\ell (k+h_j)$, $k\in \N$, where  $\ell\in \N$, $h_1,\ldots, h_\ell\in \Z^+$, and $w$ is as before,  we have    for every sequence $a\colon\N\to \S$ that
$$	
\sup_{d, n\in \N}\Big|\sum_{k=1}^n a(dQ(k)) \, w(k)\Big|=+\infty.
$$
But our methods do not allow us to deal with the  non-weighted version  (where $w=1$) even
when $Q(k)=k(k+1)$, $k\in\N$.

\item The zero entropy assumption cannot be removed. To see this, let $a(k)=f(k)$ and  $w(k)=(-1)^k \overline{f(k)}$, $k\in \N$,  where $f\colon \N\to \{-1,1\}$ is any multiplicative function that satisfies the Elliott conjecture, in which case $w$ has vanishing
 self-correlations and is totally ergodic (in fact Bernoulli). Also, the assumption that the self-correlations of $w$ vanish cannot be removed. To see this, let $a=1$ and  $w(k)=e(k\alpha)$, $k\in \N$, where $\alpha$ is irrational. On the other hand,
it is not clear whether  the  assumption of total ergodicity can  be removed.
\end{itemize}
\end{remarks}
\begin{corollary}\label{C:EDP1}
	Let  $a\colon \N\to \S$ be a sequence, $\phi\colon\T\to \mathbb{U}$ be Riemann integrable with
$\int \phi=0$ and $\int |\phi|\neq 0$, and let $P\colon\R\to \T$ be a polynomial with degree at least $2$ and irrational leading coefficient.
 Then
$$
	\sup_{d,n \in \N}\Big|\sum_{k=1}^n a(dk)\,  \phi(P(k))\Big|=+\infty.
$$	
\end{corollary}

 It follows that for $l\geq 2$  and $\alpha$ irrational, the sequence   $(e(k^l\alpha))_{k\in\N}$ and  the sequence that assigns values $-1, 0,$ or $1$ according to whether   $\{k^l\alpha\}$ is in the interval  $[0,1/3)$, $[1/3,2/3)$, or $[2/3,1)$,  are good weights.

The proof of Theorem~\ref{T:EDP1} has a few interesting features.  Unlike the proof of Theorem~\ref{T:Tao} in \cite{T16},  we are not using explicitly or implicitly results from \cite{MR15,MRT15,T15} on averages of multiplicative functions in short intervals,
and also we do not carry out a separate analysis in the case where the sequence $(a(k))_{k\in\N}$ is a pretentious multiplicative function.
 To compensate for this, our argument crucially uses  the following
 ergodic result that was proved in  \cite{FH19} using a combination of ergodic theory and number theory tools developed in
 \cite{FH18} and \cite{TT18}
  (the notions involved are defined in Section~\ref{S:ergodic}):
\begin{theorem}[F., Host \cite{FH19}]\label{T:Disjoint}
All  Furstenberg systems of  a multiplicative function with values on $\mathbb{U}$ are disjoint from all zero entropy totally ergodic systems.
\end{theorem}
To get a sense of why  Theorem~\ref{T:Disjoint} is useful, we note that it implies (via Proposition~\ref{P:zero'} below)  that if $w$ is a totally ergodic sequence with zero entropy and
$f\colon \N\to\mathbb{U}$  is a multiplicative function, then the self-correlations of the sequence $f\cdot w$ split
into a product of the self-correlations of $f$ and the self-correlations of $w$. Hence, if we assume that $w$ has vanishing self-correlations, then the same holds for $f\cdot w$,
and this property implies
 Theorem~\ref{T:EDP1} (see Proposition~\ref{P:corrzero2}).

Lastly, we give examples of good weights that are given by random sequences.
 The first result applies to  independent symmetric random variables and its proof  is rather elementary.
	\begin{theorem}\label{T:Random1}
	Let $(X_k(\omega))_{k\in\N}$ be a sequence of independent random variables with $\P(X_k=-1)=\P(X_k=1)=\frac{1}{2}$, $k\in\N$, and $a\colon \N\to \mathbb{U}$ be a non-null sequence.
	Then
	$\omega$-almost surely the sequence $(a(k)\, X_k(\omega))_{k\in\N}$ is a good weight for the Erd\H{o}s discrepancy problem.
\end{theorem}
The second result applies  to  independent random variables that are not necessarily symmetric as long as they take a fixed non-zero complex value not too rarely. Its proof, due to M.~Kolountzakis, is simple, but
 makes essential use of Theorem~\ref{T:Tao} (via the criterion given in Lemma~\ref{L:goodEDP} below).

	\begin{theorem}\label{T:Random1'}
Let $(X_k(\omega))_{k\in\N}$ be a sequence of independent, complex valued,   random variables.  Suppose that for some  $c\in \C\setminus\{0\}$ the sequence  $\rho_k:=\P(X_k=c)$, $k\in\N$, is decreasing and satisfies $\sum_{k\in \N}\rho_k^l=+\infty$ for every $l\in \N$.
Then
	$\omega$-almost surely the sequence $(X_k(\omega))_{k\in\N}$ is a good weight for the Erd\H{o}s discrepancy problem.
	\end{theorem}
\begin{remark}
The assumption of monotonicity cannot be removed. To see this, take $\P(X_k=1)=1$ if $k$ is prime, and $\P(X_k=0)=1$ for all other $k\in \N$, and let $a\colon \N\to \{-1, 1\}$ be a completely multiplicative function that is equal to $(-1)^n$ on the $n$-th prime. Then $\omega$-almost surely we have $\sup_{d, n\in \N}\big|\sum_{k=1}^n a(dk) \, X_{k}(\omega)\big|\leq 1$.
\end{remark}
If we take $c=1$ and  decreasing  $\rho_k$ such that  $\rho_k\geq  \frac{1}{\log{k}}$  and
 $\P(X_k=0)=1-\rho_k$ for  $k\geq 2$, then Theorem~\ref{T:Random1'} applies, and gives  that the indicator functions of certain sparse random subsets of the integers are good weights for the Erd\H{o}s discrepancy problem.

\subsection{Results related to weighted sums of multiplicative functions}\label{S:WS}
As was the case in the proof of Theorem~\ref{T:Tao} in  \cite{T16}, the unboundedness of weighted discrepancy sums for arbitrary unit modulus sequences
follows from similar unboundedness properties of unit modulus  completely multiplicative functions. We state next some related results that are of independent interest.
 \begin{theorem}\label{T:mf1}
	Let  $f\colon \N\to \mathbb{U}$ be a  non-null multiplicative function and   $w\colon \N\to \mathbb{U}$   be non-null,  totally ergodic,  with   zero entropy, and  vanishing self-correlations.
 Then
\begin{equation}\label{E:fj'}
	 \sup_{n\in  \N}\Big|\sum_{k=1}^n f(k)\,  w(k)\Big|=+\infty.
\end{equation}
In fact, the following stronger property holds:  If  $w$ is as before,  $f_1,\ldots, f_\ell\colon \N\to \mathbb{U}$ are multiplicative functions, and   $h_1,\ldots, h_\ell\in \Z^+$ are such that  the sequence
$( \prod_{j=1}^\ell f_j(k+h_j))_{k\in\N}$ is non-null, then we have
\begin{equation}\label{E:fj''}
\sup_{n\in \N}\Big|\sum_{k=1}^n \prod_{j=1}^\ell f_j(k+h_j)\,  w(k)\Big|=+\infty.
\end{equation}
\end{theorem}
\begin{remark}
 Note that for $w=1$ although   \eqref{E:fj'} holds for all completely multiplicative functions with values on $\S$, it fails for some non-null multiplicative functions with  values on $\mathbb{U}$. For instance it  fails for  $f(k)=(-1)^{k+1}$, $k\in\N$, and for all non-trivial Dirichlet characters.
\end{remark}

Regarding the non-weighted version of \eqref{E:fj''}, not much is known for $\ell\geq 2$. For instance,  it is not  known whether for every completely multiplicative function $f\colon \N\to \S$  we have
$$
\sup_{n\in\N}\Big|\sum_{k=1}^n f(k)\, f(k+1)\Big|=+\infty.
$$
This problem
was raised by J.~Ter\"av\"ainen and A.~Klurman, who  remarked that it is not even clear how to prove that
$$
\limsup_{n\to \infty}\Big|\sum_{k=1}^n \lambda(k)\, \lambda(k+1)\Big|\geq 5
$$
where $\lambda$ is the Liouville function.
On the other hand, it is an immediate consequence of the next corollary, that if $f\colon \N\to \S$ is a multiplicative function, $l\geq 2$, and $\alpha$ is irrational, then   we have
$$
\sup_{n\in\N}\Big|\sum_{k=1}^n f(k)\, f(k+1) \, e(k^l \alpha)\Big|=+\infty.
$$

\begin{corollary}\label{C:mf1}
	Let $\phi\colon\T\to \mathbb{U}$ be a Riemann integrable function with
$\int \phi=0$ and $\int |\phi|\neq 0$,
 and $P\colon\R\to \T$  be a polynomial with degree at least $2$ and irrational  leading  coefficient. Then for all multiplicative functions
  $f_1,\ldots, f_\ell\colon \N\to \mathbb{U}$ and   $h_1,\ldots, h_\ell\in \Z^+$ such that  the sequence
$( \prod_{j=1}^\ell f_j(k+h_j))_{k\in\N}$ is non-null, we have
$$
\sup_{n\in \N}\Big|\sum_{k=1}^n \prod_{j=1}^\ell f_j(k+h_j)\,  \phi(P(k))\Big|=+\infty.
$$
\end{corollary}

 Regarding weights given by random $\pm 1$ sequences, we have the following result:
\begin{theorem}\label{T:Random2}
Let $(X_k(\omega))_{k\in\N}$ be a sequence of independent random variables with $\P(X_k=-1)=\P(X_k=1)=\frac{1}{2}$, $k\in\N$.
Then
$\omega$-almost surely the following holds: For every  $\ell\in \N$, all     multiplicative functions $f_1,\ldots, f_\ell \colon \N\to \mathbb{U}$, and  $h_1,\ldots, h_\ell \in\Z^+$
 such that  the sequence
$( \prod_{j=1}^\ell f_j(k+h_j))_{k\in\N}$ is non-null, we have
	\begin{equation}\label{E:Xn}
	\sup_{n\in  \N}\Big|\sum_{k=1}^n \prod_{j=1}^\ell f_j(k+h_j)\,  X_k(\omega)\Big|=+\infty.
	\end{equation}
	 \end{theorem}
 \begin{remarks}
 $\bullet$ It is not hard to show that for any fixed collection of arbitrary sequences $f_1,\ldots, f_\ell\colon \N\to \mathbb{U}$, we have that  \eqref{E:Xn} holds $\omega$-almost surely.
So  	the important point in Theorem~\ref{T:Random2}  is that the set of $\omega$'s for which the conclusion holds is independent of the (uncountably many) multiplicative functions $f_1,\ldots, f_\ell$.

$\bullet$  For $\ell=1$, Theorem~\ref{T:Random1} gives better results that apply to  not necessarily symmetric  random variables. But for $\ell\geq 2$ the method of proof of Theorem~\ref{T:Random1} fails to give \eqref{E:Xn} (since the relevant unweighted result is not known).
 \end{remarks}
Theorem~\ref{T:Random2}  is based on Theorem~\ref{T:Random3} below, which is  proved  by combining some simple counting arguments and concentration of measure estimates for sums of independent random variables.

\subsection{Proof strategy} Let us first recall the  proof strategy of Theorem~\ref{T:Tao} given in \cite{T16}.
 An immediate consequence of Theorem~\ref{T:Tao}  is that for every completely multiplicative function  $f\colon \N\to \S$ we have
\begin{equation}\label{E:Tao1}
		\sup_{ n\in \N}\Big|\sum_{k=1}^n f(k) \Big|=+\infty.\footnote{See also \cite{K17} for a classification  of  multiplicative functions (but not necessarily completely multiplicative) with values $\pm 1$  that satisfy \eqref{E:Tao1}.}
		\end{equation}
It turns out that a variant of this special case (see Proposition~\ref{P:Reduction} below for $w=1$) is the key ingredient in the proof of Theorem~\ref{T:Tao}.
The proof of \eqref{E:Tao1} given in \cite{T16} proceeds by considering separately the case where $f$ is structured (``pretentious'')
and random (``non-pretentious''). The latter case  can be treated
(as in Proposition~\ref{P:corrzero} below) using  the  identities
\begin{equation}\label{E:Tao2}
\lE_{n\in\N}\,  f(n+h)\, \overline{f(n)}=0, \quad h\in \N,
\end{equation}
which hold for random-like (``non-pretentious'') multiplicative functions.

Likewise, our arguments  rely on weighted variants of \eqref{E:Tao1} and \eqref{E:Tao2}
that are of independent interest. For instance, we prove that if $l\geq 2$ and $\alpha$ is irrational, then for every  multiplicative function  $f\colon \N\to \S$ we have
\begin{equation}\label{E:Tao1'}
	\sup_{ n\in \N}\Big|\sum_{k=1}^n f(k) \, e(k^l\alpha)\Big|=+\infty,
	\end{equation}
and we also prove stronger results involving weighted sums of products of shifts of several multiplicative functions.
To prove \eqref{E:Tao1'} we rely on one of the main  results in \cite{FH19}, which implies that for every $l\in \N$ and $\alpha$ irrational we have
\begin{equation}\label{E:Tao2'}
\lE_{n\in\N}\,  f(n+h)\, \overline{f(n)}\, e(n^l\alpha)=0.
\end{equation}
The fact that \eqref{E:Tao2'} holds for every multiplicative function   $f\colon \N\to \S$ (which is not true for \eqref{E:Tao2})
simplifies the proof of \eqref{E:Tao1'}, versus the argument given for the proof of  \eqref{E:Tao1} in \cite{T16},  and ultimately of the fact that $(e(k^l\alpha))_{k\in\N}$ is a good weight.
 One reason is that
we do not have to carry out a separate analysis in the case where $f$ is structured  (``pretentious''), as was the case in \cite{T16}.

The proofs of the results  concerning  random weights are simpler. Theorem~\ref{T:Random1} is based on a variant of \eqref{E:Tao2'} that uses random weights and is proved in Theorem~\ref{T:Random3} via elementary techniques. Theorem~\ref{T:Random1'}
is deduced from Theorem~\ref{T:Tao} using an elementary argument given in  Section~\ref{SS:1.6}.

\subsection{Some open problems}
A possible strengthening of Theorem~\ref{T:Tao} is given in the following problem (for $w=1$ and $a=b$ the problem was previously proposed  by  J.~Ter\"av\"ainen and A.~Klurman
at the December 2018 workshop  of the American Institute of Mathematics  ``Sarnak's Conjecture''):
\begin{problem}
	Is it true that for every   $a,b\colon  \N\to \S$  we have
	$$
	\sup_{d, n\in \N}\Big|\sum_{k=1}^n  a(dk) \, b(d(k+1)) \,  w(k)\Big|=+\infty
	$$
when $w(k)=1$, $k\in\N$,  or when $w(k)=e(k^2\alpha)$, $k\in\N$,  with $\alpha$ irrational?
\end{problem}
When $w=1$ the problem is open even when $a=b=f$, where $f\colon \N\to \S$ is a completely multiplicative function (see remarks on Section~\ref{S:WS}).
More generally, one can  ask whether  for the previous choices of the sequence $w$, for every   $a_1,\ldots, a_\ell\colon \N\to \S$ and   all $h_1,\ldots, h_\ell\in \Z^+$ we have
	$$
	\sup_{d, n\in \N}\Big|\sum_{k=1}^n \prod_{j=1}^\ell a_j(d\, (k+h_j)) \,  w(k)\Big|=+\infty.
	$$
Corollary~\ref{C:mf1} shows that the answer is yes when 
$a_1,\ldots,  a_\ell$ are  multiplicative functions with values on $\S$ and $w$ is the sequence $(e(k^2\alpha))_{k\in\N}$ with $\alpha$ irrational. But unlike the previous discrepancy statements,
  we do not have a way to reduce Problem 1 to one about weighted sums of  multiplicative functions.
Any such reduction probably depends upon obtaining an  integral representation result, analogous to Proposition~\ref{P:Bochner} below, for sequences of the form
$A(k_1,\ldots, k_\ell)=\E_{d\in \Phi} \prod_{j=1}^\ell a_j(dk_j)$, $k_1,\ldots,k_{\ell} \in\N$, where $\Phi$ is a multiplicative F\o lner sequence (see Section~\ref{SS:Folner}) along which all previous averages exist. Note though that more complicated ``higher order multiplicative functions'' arise this way, for instance,  if $f\colon \N\to \S$ is defined by $f(k)=e((n_1\alpha_1+\cdots+n_l\alpha_l)^2)$, where $k=p_1^{n_1}\cdots p_l^{n_l}$  is the unique factorization of $k\in \N$,  and $\alpha_1,\ldots, \alpha_l\in \R$, then
$$
f^2(k)=\E_{d\in \Phi}  f(d)\, \overline{f^2(dk)}\,  f(dk^2)
$$
for every $k\in \N$.

On a different direction, it seems likely that the zero integral condition in Corollary~\ref{C:EDP1}
can be removed. Proving this would probably necessitate to
combine arguments of this article
with a detailed analysis of the pretentious case (similar to the one in \cite{T16}), and it is not clear how to do this.
\begin{problem}\label{P:1}
	Is it true  that Corollary~\ref{C:EDP1} holds even if we do not assume that $\int \phi=0$?
\end{problem}
Let us say that a subset $S$ of $\N$ is {\em good for the Erd\H{o}s discrepancy problem}, or simply, {\em good}, if the
indicator function ${\bf 1}_S$ is a good weight  for the Erd\H{o}s discrepancy problem.
By taking the sequence $(a(k))_{k\in\N}$ in \eqref{E:WEDP}  to  be an appropriate multiplicative function one easily verifies that the sets $\{n\not\equiv 0\pmod{r}\}$ for $r\geq 3$,  $\{2^n,n\in \N\}$, and  $\{p_n, n\in\N\}$, where $p_n$ is the $n$-th prime,  are bad.   On the other hand, it is easy to deduce form Theorem~\ref{T:Tao} that the sets $r\Z$ for $r\in \N$ and $\{n^l,n\in\N\}$ for $l\in\N$,  are good. But it is not  at all clear whether certain simple sets that lack  multiplicative structure are good.
\begin{problem}
 Are the sets $\{p_n+ 1, n\in \N\}$, $\{n^2\pm 1, n\in\N\}$,   $\{2^n+1, n\in\N\}$,
or  $\{[n^c], n\in\N\}$ for $c> 1$ not an integer,  good for the Erd\H{o}s discrepancy problem?
\end{problem}
Theorem~\ref{T:Random1'} implies that random subsets of the integers with  positive density,
and certain sparse random subsets with density roughly $(\log{N})^{-1}$ in $[N]$,
 are almost surely good. But how about  sparser random subsets?
\begin{problem}
Let $a\in (0,1]$ and  $(X_k(\omega))_{k\in\N}$ be a sequence of independent random variables with $\P(X_k=1)=k^{-a}$, $\P(X_k=0)=1-k^{-a}$, $k\in\N$.
Is it true that $\omega$-almost surely the sequence $(X_k(\omega))_{k\in\N}$
is a good weight for the Erd\H{o}s discrepancy problem?
\end{problem}

\subsection{Notation}\label{SS:def}
With $\mathbb{U}$ we denote the complex unit disc $\{z\in \C\colon |z|\leq 1\}$ and with $\S$ we denote the complex unit circle $\{z\in \C\colon |z|=1\}$.
 With $\T$  we denote the  $1$-dimensional torus  that we identify with $\R/\Z$. With $\N$ we denote the positive integers and with $\Z^+$ the non-negative integers. For $N\in \N$ we let $[N]:=\{1,\ldots, N\}$. For $t\in \R$ we also let $e(t):=e^{2\pi i t}$.

 If $A$ is a non-empty finite subset of $\N$  we let
$$
\E_{n\in A}\,a(n):=\frac{1}{|A|}\sum_{n\in A}\, a(n), \quad
\lE_{n\in A}\,a(n):=\frac{1}{\sum_{n\in A}\frac{1}{n}}\sum_{n\in A}\frac{a(n)}{n}.
$$
If $A$ is an infinite subset of $\N$ we let
$$
\E_{n\in A}\, a(n):=\lim_{N\to\infty} \E_{n\in A\cap [N]}\, a(n), \quad
\lE_{n\in A}\, a(n):=\lim_{N\to\infty} \lE_{n\in A\cap [N]}\, a(n)
$$
whenever  the limits exist.

 With $\bN= ([N_l])_{l\in\N}$ we denote a  sequence of intervals with $N_l\to \infty$.
We let
$$
\E_{n\in\bN}\, a(n):=\lim_{l\to\infty}\E_{n\in[N_l]} \, a(n), \quad
\lE_{n\in\bN}\, a(n):=\lim_{l\to\infty}\lE_{n\in[N_l]} \, a(n)
$$
whenever   the limits exist. Using partial summation one sees that if
$\E_{n\in \N}\,a(n)=0$, then also $\lE_{n\in \N}\,a(n)=0$ (but the converse does not hold in general).

\section{Reduction to statements about multiplicative functions}
\subsection{Multiplicative averages}\label{SS:Folner} We denote by  $\Q^+$  the multiplicative group of positive rationals.

\begin{definition} We say that $\Phi=(\Phi_N)_{N\in\N}$  is a \emph{multiplicative F\o lner sequence}, if $\Phi_N$ is a finite subset of $\N$ for every $N\in \N$, and  for every $r\in\Q^+$ we have
\begin{equation}
\label{E:Folner}
\lim_{N\to \infty} \frac 1{|\Phi_N|}{|(r\inv\Phi_N)\triangle  \Phi_N|}=0
\end{equation}
where $r\inv \Phi_N:=\{n\in\N\colon rn\in \Phi_N\}$.
\end{definition}
An example of a multiplicative F\o lner sequence is given by
$$
\Phi_N:=
\{p_1^{k_1}\cdots p_N^{k_N}\colon 0\leq k_1,\ldots, k_N\leq N\}, \quad N\in\N,
$$
where $(p_n)_{n\in\N}$ denotes the sequence of primes.
\begin{definition}  If $ \Phi=(\Phi_N)_{N\in\N}$ is  a multiplicative F\o lner sequence and
 $a\colon \N\to \mathbb{U}$ is such that the average below exists,  we define the {\em multiplicative   average of the sequence $a$ along $ \Phi$} by
$$
\E_{n\in \Phi}\, a(n):=\lim_{N\to \infty} \E_{n\in\Phi_N}\, a(n).
$$
\end{definition}
Note that property \eqref{E:Folner} implies the following dilation invariance property of the multiplicative averages:  For every $a\colon \Q^+\to \mathbb{U}$,
multiplicative F\o lner sequence $\Phi$, and    $r\in \Q^+$, we have
\begin{equation}\label{E:dilation}
\E_{n\in  \Phi}\,  (a(rn)-   a(n))=0.
 \end{equation}

\subsection{Reduction to multiplicative functions via Bochner's theorem}
 A variant of the next lemma was proved in \cite[Section~2]{T16} using Fourier analysis on an appropriate finite Abelian group (of the form $(\Z/M\Z)^r$ for large $M,r\in\N$) and a compactness argument. We use a  somewhat different approach (also used in \cite[Section~10.2]{FH15}) that invokes Bochner's theorem on positive definite functions. We first introduce some notation.
 \begin{definition}
 With $\CM$ we denote the set of all completely multiplicative functions $f\colon \N\to \S$.
 \end{definition}
 Endowed with  pointwise multiplication and the topology of pointwise convergence,
the set $\CM$   is a compact (metrizable) Abelian
group.
\begin{proposition}\label{P:Bochner}
  Let $A\colon \N^2\to \C$ be defined by
$$
A(k,l):=\E_{d\in \Phi} \, a(dk) \, \overline{a(dl)}, \quad k,l\in\N,
$$
where $a\colon \N\to \C$ is a bounded sequence and  $ \Phi=(\Phi_N)_{N\in\N}$ is a multiplicative F\o lner sequence such that all the averages
above exist.  Then  there exists a  (positive) measure $\sigma$ on the space $\CM$, with total mass equal to $\E_{d\in  \Phi}|a(d)|^2$,   such that
$$
A(k,l)=\int_{\CM} f(k) \, \overline{f(l)} \, d\sigma(f), \quad k,l\in \N.
$$
\end{proposition}
\begin{proof}
We first   extend the sequence $a$ to the positive rationals $\Q^+$ by letting $a(r)=0$ for $r\in \Q^+\setminus \N$. We define $B\colon \Q^+\to \C$ as follows
$$
B(r):=\E_{d\in \Phi} \, a(rd) \, \overline{a(d)}, \quad r\in\Q^+.
$$
Using the dilation invariance property \eqref{E:dilation}  and our assumption that the averages defining the sequence $A$ exist, we deduce that the  averages below exist and we have
$$
B(rs^{-1})=\E_{d\in \Phi} \, a(rd) \, \overline{a(sd)}, \quad r,s\in\Q^+.
$$

We are going to use  this identity in order to  verify that $B$ is a positive definite sequence on $\Q^+$ with pointwise multiplication. Indeed, for all $c_1,\ldots, c_N\in \C$ and $r_1,\ldots, r_N \in \Q^+$,  we have
$$
\sum_{i,j\in [N]}\, c_i\, \overline{c_j}\, B(r_i r_j^{-1})=
\E_{d\in \Phi} \big|\sum_{i\in [N]}\, c_i \ a(r_id)\big|^2
\geq 0.
$$
  Note that the dual group of $(\Q^+, \cdot)$ consists of the  completely multiplicative functions on $\Q^+$ with unit modulus,  and any such $\psi\colon \Q^+\to \S$ satisfies
  $\psi(m/n)=f(m)\, \overline{f(n)}$, $m,n\in\N$,  for some completely multiplicative function $f\in \CM$.
  A well known theorem of   Bochner
  gives  that there exists a (positive) Borel measure $\sigma$ on the space $\CM$   such that
$$
B(k/l)=\int_{\CM} f(k) \, \overline{f(l)} \, d\sigma(f), \quad k,l\in \N.
$$
The total mass of $\sigma$ is $B(1)=\E_{d\in  \Phi} |a(d)|^2$.
  Lastly, we have
  $$
  B(k/l)=\E_{d\in \Phi} \, a(kd/l) \, \overline{a(d)}=\E_{d\in \Phi} \, a(kd) \, \overline{a(ld)},
  $$and  the proof is complete.
\end{proof}
Using the previous representation theorem we get the following criterion:
\begin{proposition}\label{P:Reduction}
Let $w\colon \N\to \mathbb{U}$ be   such that for every probability measure
$\sigma$ on the space $\CM$  we have
$$
\sup_{n\in\N}\int_\CM\Big|\sum_{k=1}^n f(k) \, w(k)\Big|^2\ d\sigma(f)=+\infty.
$$
Then $w$ is a good weight for the Erd\H{o}s discrepancy problem.
\end{proposition}
\begin{proof}
Arguing by contradiction, suppose that $w$ is not a good weight for the Erd\H{o}s discrepancy problem. Then  there exists a sequence $a\colon \N\to \S$ such that
$$
	\sup_{d, n\in \N}\Big|\sum_{k=1}^n a(dk) \, w(k)\Big|<+\infty.
	$$
We  average with respect to  $d$ over a multiplicative
F\o lner sequence of intervals $ \Phi=(\Phi_N)_{N\in\N}$, chosen so that all relevant averages below exist (such a sequence can always be found using a diagonalisation argument), and deduce that
\begin{equation}\label{E:adk}
\sup_{n\in\N}	\E_{d\in \Phi} \Big|\sum_{k=1}^n a(dk) \, w(k)\Big|^2<+\infty.
\end{equation}
Expanding the square we get that the expression in \eqref{E:adk} is equal to
\begin{equation}\label{E:mnN}
\sup_{n\in\N}\Big(\sum_{k,l\in [n]}	 \, w(k) \, \overline{w(l)}  \, A(k,l)\Big)
	\end{equation}
where
$$
A(k,l):=\E_{d\in \Phi} \, a(dk) \, \overline{a(dl)}, \quad k,l\in\N.
$$
By Lemma~\ref{P:Bochner}, there exists a  (positive) measure $\sigma$ on the space $\CM$, with total mass $\E_{d\in \Phi } |a(d)|^2=1$, such that
$$
A(k,l)=\int_{\CM} f(k) \, \overline{f(l)} \, d\sigma(f), \quad k,l\in \N.
$$
We deduce that the  expression \eqref{E:mnN}, and hence the expression in \eqref{E:adk},
 is equal to
$$
\sup_{n\in\N}	\int_\CM \Big|\sum_{k=1}^n f(k) \, w(k)\Big|^2\, d\sigma(f).
$$
Hence,
$$
\sup_{n\in\N}	\int_\CM \Big|\sum_{k=1}^n f(k) \, w(k)\Big|^2\, d\sigma(f)<+\infty.
$$
This contradicts our assumption and completes the proof.
\end{proof}
\subsection{Reduction to correlation estimates} As was the case in  \cite{T16}, a key step  in the proof of our main results is  an  elementary observation
that allows  to deduce unboundedness of partial sums from vanishing of
self-correlations (which are defined  using logarithmic averages because of reasons explained in the next section).
\begin{proposition}\label{P:corrzero}
Let $b\colon \N\to \mathbb{U}$ be a  non-null sequence such that for every $h\in \N$ we have
$$
 \lE_{n\in  \N}\, b(n+h)\, \overline{b(n)}=0.
$$
Then
$$
\sup_{n\in\N}\Big|\sum_{k=1}^n   b(k)\Big|=+\infty.
	$$
\end{proposition}
\begin{proof}
Arguing by contradiction, suppose that the conclusion fails. Then there exists $C>0$ such that
$$
\sup_{n\in\N}\Big|\sum_{k=1}^n b(k)\Big|\leq C.
$$
Using this, we can find
 a sequence of intervals $\bN=([N_l])_{l\in\N}$, with  $N_l\to\infty$,  such that all  averages $\lE_{n\in  \bN}$ written  below exist and for every $H\in \N$
 we have
$$
	 \lE_{n\in  \bN}\Big|\sum_{h=1}^{H}   b(n+h)\Big|^2=\lE_{n\in  \bN}\Big|\sum_{k=1}^{n+H}   b(k)-\sum_{k=1}^n   b(k)\Big|^2\leq 4C^2.
	$$
Since the sequence $b$ is non-null, we have
$$
B:=\lE_{n\in  \bN}|b(n)|^2>0.
$$
Next, notice that
$$
 \lE_{n\in  \bN}\Big|\sum_{h=1}^{H}   b(n+h)\Big|^2=
  \sum_{1\leq h_1\neq h_2\leq H} \lE_{n\in  \bN} \,  b(n+h_1)\, \overline{b(n+h_2)}+H B=H B
$$
since by our assumption $\lE_{n\in  \bN} \,  b(n+h_1)\, \overline{b(n+h_2)}=0$ for $h_1\neq h_2$ and we also used
twice that the  logarithmic averages of a bounded sequence are translation invariant.
From the above we deduce that $H B\leq 4C^2$ and we get a contradiction by choosing $H>4C^2/B$.

\end{proof}
\begin{proposition}\label{P:corrzero2}
	Let $w\colon \N\to \mathbb{U}$ be a non-null sequence such that for every multiplicative function $f\colon \N\to \S$ and every  $h\in \N$    we have
	$$
	\lE_{n\in  \N}\, (f\cdot w)(n+h)\, \overline{(f\cdot w)(n)}\, =0.
	$$
Then $w$ is a good weight for the Erd\H{o}s discrepancy problem.
\end{proposition}

\begin{proof}
	Arguing by contradiction, suppose that the conclusion fails. Then by Proposition~\ref{P:Reduction}  there  exist a sequence $w\colon \N\to \mathbb{U}$, a probability measure $\sigma$ on the space $\CM$, and  $C>0$, such that
		$$
	\sup_{n\in\N}\int_{\CM}\Big|\sum_{k=1}^n    f(k)\, w(k) \Big|^2\, d\sigma(f)\leq C.
	$$
 Using this and a diagonalization argument, we can find  a sequence of intervals $\bN=([N_l])_{l\in\N}$, with  $N_l\to\infty$,  such that $\lE_{n\in  \bN}|w(n)|^2$  and all  averages $\lE_{n\in  \bN}$ written  below exist and for every $H\in \N$
	we have
\begin{multline}\label{E:1}
	\lE_{n\in  \bN}\int_{\CM}\Big|\sum_{h=1}^{H}   (f\cdot w)(n+h)\Big|^2 d\sigma(f) = \lE_{n\in  \bN}\int_\CM \Big|\sum_{k=1}^{n+H}   (f\cdot w)(k)-\sum_{k=1}^n   (f\cdot w)(k)\Big|^2 d\sigma(f)  \\
\leq \lE_{n\in  \bN}\int_\CM 2\, \Big(\Big|\sum_{k=1}^{n+H}   (f\cdot w)(k)\Big|^2+\Big|\sum_{k=1}^n   (f\cdot w)(k)\Big|^2\Big)\,  d\sigma
 \leq 4C.
	\end{multline}

We let
	$$
	A:=\lE_{n\in  \bN}|w(n)|^2>0
	$$
	where the positiveness follows since the sequence  $w$ is non-null by our assumption.
	Next, notice that
	$$
	\lE_{n\in  \N}\Big|\sum_{h=1}^{H}   (f\cdot w)(n+h)\Big|^2=
	\sum_{1\leq h_1\neq h_2\leq H} \lE_{n\in  \N} \,  (f\cdot w)(n+h_1)\, \overline{(f\cdot w)(n+h_2)}+H A=H A
	$$
	since by our assumption $\lE_{n\in  \N} \,  (f\cdot w)(n+h_1)\, \overline{(f\cdot w)(n+h_2)}=0$ for $h_1\neq h_2$. Since $\sigma$ is a probability measure,   we deduce using the bounded convergence theorem that
	\begin{equation}\label{E:2}
	\lE_{n\in  \bN}\int_{\CM}\Big|\sum_{h=1}^{H}   (f\cdot w)(n+h)\Big|^2\ d\sigma(f)=H A.
	\end{equation}

Combining \eqref{E:1} and \eqref{E:2} we deduce	
 that $H\, A\leq 4C$ and we get a contradiction by choosing $H>4C/A$.
\end{proof}

\section{Notions and results from ergodic theory}\label{S:ergodic}
The proof of our main  results regarding structured (zero entropy) sequences depend on some notions and results in ergodic theory that we describe next. The material in this section is not needed for the results concerning random weights.
\subsection{Measure preserving systems}\label{SS:mps}
A {\em measure preserving system}, or simply {\em a system}, is a quadruple $(X,\CX,\mu,T)$
where $(X,\CX,\mu)$ is a probability space and $T\colon X\to X$ is an invertible, measurable,  measure preserving transformation.
We typically omit the $\sigma$-algebra $\CX$  and write $(X,\mu,T)$. Throughout,  for $n\in \N$ we denote  by $T^n$   the composition $T\circ  \cdots \circ T$ ($n$ times) and let $T^{-n}:=(T^n)^{-1}$ and $T^0:=\id_X$. Also, for $f\in L^1(\mu)$ and $n\in\Z$ we denote by  $T^nf$ the function $f\circ T^n$.

We say that the system $(X,\mu,T)$ is {\em ergodic}  if the only functions $f\in L^1(\mu)$ that
satisfy $Tf=f$ are the constant ones. It is {\em totally ergodic} if $(X,\mu,T^d)$ is ergodic for every $d\in \N$.

\subsection{Furstenberg systems}\label{SS:ergodicnotation}
For readers convenience, we reproduce here some  ergodic notions and constructions that can also be found in
\cite{FH18, FH19}.
For the purposes of this  article,
all averages in the definitions below are taken to be logarithmic. The reason is that we later on invoke results from ergodic theory, like Theorem~\ref{T:DisjointSeveral} below, that are only known when the joint Furstenberg systems are defined using logarithmic averages. This limitation comes from the number theoretic input used in the proof of Theorem~\ref{T:DisjointSeveral}, in particular, the identities in \cite[Theorem~3.1]{FH19}.
\begin{definition}\label{D:correlations}
	Let $ \bN:=([N_l])_{l\in\N}$ be a sequence of intervals with $N_l\to \infty$.
	We say that a finite collection of bounded sequences $\mathcal{A}=\{a_1,\ldots, a_\ell\}$
	{\em admits log-correlations  on $\bN$}, if the   limits
	$$
	\lim_{l\to\infty} \lE_{n\in [N_l]}\,  \prod_{j=1}^m \tilde{a}_j(n+h_j)
	$$
	exist for all $m \in \N$, all   $ h_1,\ldots, h_m\in \Z$, 
	and all $\tilde{a}_1,\ldots,\tilde{a}_m\in \mathcal{A}\cup \overline{\mathcal{A}}$.
\end{definition}
For  every finite collection of sequences that admits log-correlations on a given sequence of intervals,  we use a  variant of the correspondence principle of Furstenberg~\cite{Fu77, Fu} in order
to associate a measure preserving system that captures the statistical properties of these sequences.
\begin{definition}\label{P:correspondence}
	Let $a_1,\ldots, a_\ell\colon \Z\to\mathbb{U}$ be sequences that  admit
	log-correlations  on the sequence of intervals
	$\bN:=([N_l])_{l\in\N}$. We let  $\mathcal{A}:=\{a_1,\ldots,a_\ell\}$, $X:=(\mathbb{U}^\ell)^\Z$,  $T$ be the shift transformation on $X$,
	and $\mu$ be the weak-star limit of the sequence of measures $(\lE_{n\in[N_l]} \, \delta_{T^n a})_{l\in \N}$
	where $a:=(a_1,\ldots, a_\ell)$ is thought of as an element of $X$. We call $(X,\mu,T)$
	the {\em joint Furstenberg system
		associated with} ($\mathcal{A}$, $\bN$).
\end{definition}
\begin{remark}
	If we are given sequences $a_1,\ldots, a_\ell\colon \N\to \mathbb{U}$ that are defined on $\N$, we extend them to $\Z$ in an arbitrary way. It is easy to check that  the measure $\mu$ will  not depend on the extension.
\end{remark}
Note that a collection of sequences $a_1,\ldots, a_\ell\colon \Z\to \mathbb{U}$ may have several non-isomorphic joint Furstenberg systems
depending on which sequence of intervals $\bN$ we use in the evaluation of  their joint correlations. For convenience of exposition, we sometimes associate  a property of ergodic nature with a given finite collection of sequences if all joint Furstenberg systems of the collection have this property. In particular, we often use the following terminology:
\begin{definition}
	We say that a  sequence $a\colon \Z\to \mathbb{U}$  is {\em totally ergodic} and/or has {\em zero entropy},  if all its Furstenberg systems are totally ergodic and/or have zero entropy.
\end{definition}
\begin{remark}
In \cite{K73}, a zero entropy sequence is called {\em completely deterministic}.
	\end{remark}
Examples of zero entropy  sequences include the sequences $(e(n^l\alpha))_{n\in\N}$ where $l\in\N$ and $\alpha\in \R$; these sequences  are also totally ergodic if $\alpha$ is irrational (see Proposition~\ref{P:ex} below).

\subsection{Disjointness properties}
We will use the following notion that was  introduced by  Furstenberg in~\cite{Fu67}:
\begin{definition} We say that two systems $(X, \mu,T)$ and $(Y, \nu,S)$
	are {\em disjoint}, if the only $T\times S$ invariant measure on the product space
	$(X\times Y, \mu\times \nu)$, with first and second  marginals the measures $\mu$ and $\nu$	respectively, is the product measure $\mu\times \nu$.
\end{definition}
The notion of disjointness in ergodic theory naturally introduces the following notion of statistical disjointness
of two finite collections of bounded sequences.
\begin{definition}
	We say that two finite collections $\mathcal{A}$ and $\mathcal{B}$  of sequences   with values on $\mathbb{U}$, are {\em statistically disjoint},  if
	all the joint Furstenberg systems of the collection  $\mathcal{A}$ are
	(measure-theoretically) disjoint form all the joint Furstenberg systems of the  collection $\mathcal{B}$.
\end{definition}
The next result shows that if  two collections of sequences are statistically disjoint, then  all their joint correlations decouple
into products of  joint  correlations of $\mathcal{A}$ and joint  correlations of $\mathcal{B}$.

\begin{proposition}\label{P:Disjoint}
	Let $\mathcal{A}=\{a_1,\ldots, a_\ell\}$ and $\mathcal{A}'=\{a'_1,\ldots, a'_{\ell'}\}$ be two collections of   sequences with values on $\mathbb{U}$ that
	are statistically disjoint.
	Then
	$$
	\lim_{N\to \infty} \big(\lE_{n\in [N]}(A_n\,  A'_n) - \lE_{n\in [N]} A_n \cdot \lE_{n\in [N]} A'_n\big)=0
	$$
	for all choices $A_n=\prod_{j=1}^m \tilde{a}_j(n+h_j)$,
	$A'_n=\prod_{j=1}^{m'} \tilde{a}'_j(n+h_j')$, $n\in\N$,
	where $m,m',h_j,h_j'\in \N$ and $\tilde{a}_j\in \mathcal{A}\cup \overline{\mathcal{A}}$,
	$\tilde{a}'_j\in \mathcal{A}'\cup \overline{\mathcal{A}'}$   are arbitrary.
\end{proposition}
\begin{proof}
	Arguing by contradiction,  suppose that the conclusion fails. Then there exists a sequence of intervals $\bN=([N_l])_{l\in\N}$, with $N_l\to \infty$,  on which the family $\mathcal{A}\cup \mathcal{A}'$
	admits log-correlations  and we have
	\begin{equation}\label{E:AB}
		\lE_{n\in \bN}(A_n\,  A'_n) \neq  \lE_{n\in \bN} A_n \cdot \lE_{n\in \bN} A'_n
	\end{equation}	
	for some choice of  $A_n=\prod_{j=1}^m \tilde{a}_j(n+h_j)$,
	$A'_n=\prod_{j=1}^{m'} \tilde{a}'_j(n+h_j')$, $n\in\N$,
	where $m,m',h_j,h_j'\in \N$  and $\tilde{a}_j\in \mathcal{A}\cup \overline{\mathcal{A}}$,
	$\tilde{a}'_j\in \mathcal{A}'\cup \overline{\mathcal{A}'}$.
	Let $(X,\mu,T)$ and $(X',\mu',T')$ be the joint  Furstenberg systems associated with  ($\mathcal{A}$, $\bN$)  and ($\mathcal{A}'$, $\bN$) respectively.

	We let  $x_0:=(a_1,\ldots, a_\ell)\in X$ and  $x_0':=(a'_1,\ldots, a'_{\ell'})\in  X'$. After passing to a  subsequence of $\bN$ (which for simplicity we denote again by $\bN$),  we can assume that the weak-star limit
	\begin{equation}\label{D:rho}
		\rho:=\lim_{l\to\infty} \lE_{n\in [N_l]}\ \delta_{(T\times T')^n(x_0, x_0')}
	\end{equation}
	exists and defines a $T\times T'$
	invariant measure
	on $X\times X'$. The projection of $\rho$ on $X$ is the weak-star limit $\lim_{l\to\infty}\lE_{n\in [N_l]}\delta_{x_0}$, which is the measure $\mu$. Likewise, the projection of $\rho$ on $X'$
	is the measure $\mu'$.
	Since the families $\mathcal{A}$ and $\mathcal{A}'$ are statistically disjoint, 
	the systems $(X,\mu,T)$ and $(X',\mu',T')$
	are  disjoint,   hence
	\begin{equation}\label{E:rho}
		\rho=\mu\times \mu'.
	\end{equation}
	
	Now for   $x=(x_1(n),\ldots, x_\ell(n))_{n\in\Z}\in X$  we let
	$$
	F_{h,j}(x):=x_j(h), \quad h\in \Z,\,  j\in \{1,\ldots, \ell\}.
	$$
	Likewise,  for  $x'=(x'_1(n),\ldots, x'_{\ell'}(n))_{n\in\Z}\in X'$ we let
	$$
	F'_{h,j}(x'):=x'_j(h), \quad h\in \Z, \, j\in \{1,\ldots, \ell'\}.
	$$
	With the above notation, we define the function   $F(x):=\prod_{j=1}^mG_{h_j,j}(x)$, $x\in X$, where for $j=1,\ldots, m$  if  $\tilde{a}_j=a_{k_j}$ or $\overline{a}_{k_j}$ for some $k_j\in \{1,\ldots, \ell\}$ we set $G_{h_j,j} $ to be  $F_{h_j,k_j} $ or
	$\overline{F}_{h_j,k_j} $ respectively. Likewise, we define the function
	$F'(x'):=\prod_{j=1}^{m'}G'_{h'_j,j}(x')$, $x'\in X'$.
	Then  using  \eqref{E:AB} and the definition of the measures $\mu,\mu'$ and the measure $\rho$ given by \eqref{D:rho},
	we get  that
	$$
	\int_{X\times X'} F(x)\, F'(x')\, d\rho(x,x') \neq \int_X F\, d\mu \cdot \int_{X'} F'\, d\mu'.
	$$
	This contradicts \eqref{E:rho} and completes the proof.
\end{proof}
The next result follows by  combining  the structural result of \cite[Theorem~1.5]{FH19} with the disjointness statement of \cite[Proposition~3.12]{FH18}.
\begin{theorem}[F., Host~\cite{FH18,FH19}]\label{T:DisjointSeveral}
	All joint Furstenberg systems of any collection of  multiplicative functions with values on $\mathbb{U}$ are disjoint from all zero entropy  totally ergodic systems.
\end{theorem}
Restating Theorem~\ref{T:DisjointSeveral} using terminology introduced in the previous definitions we get the following result:
\begin{theorem}\label{T:Disjoint'}
	Every finite collection of  multiplicative functions with values on $\mathbb{U}$ is statistically disjoint from every  totally ergodic sequence with zero entropy.
\end{theorem}

\section{Proof of main results for structured weights}

\subsection{Proof of Theorems~\ref{T:EDP1} and \ref{T:mf1}}
First we show that the assumption of Proposition~\ref{P:corrzero} is satisfied for various sequences
of interest.
\begin{proposition}\label{P:zero'}
Suppose that  $w\colon \N\to \mathbb{U}$
is a totally ergodic sequence with   zero entropy and  vanishing self-correlations. Let also  $f_1,\ldots, f_\ell\colon \N\to \mathbb{U}$ be multiplicative functions, $h_1,\ldots, h_\ell\in \Z^+$, and
  $
  b(n):=w(n)\, \prod_{j=1}^\ell f_j(n+h_j) $, $n\in\N$. Then for every $h\in\N$  we have
$$
 \lE_{n\in  \N}\, b(n+h)\, \overline{b(n)}=0.
$$
 \end{proposition}
 \begin{remark}
 For the purpose of proving Theorem~\ref{T:EDP1} we only need to consider the case where $\ell=1$ and $f_1$ is  completely multiplicative of unit modulus.
 But this special case does not seem to offer significant simplifications.
 \end{remark}
\begin{proof}
By Theorem~\ref{T:Disjoint'}, the collection of sequences $\{f_1,\ldots, f_\ell\}$ and $\{w\}$  are statistically disjoint.
By Proposition~\ref{P:Disjoint}, we have that the difference between the average
$$ \lE_{n\in  [N]}\, b(n+h)\, \overline{b(n)}
$$
and the product of averages
$$
 \lE_{n\in  [N]}\, w(n+h)\, \overline{w(n)}\cdot  \lE_{n\in  [N]}\, \prod_{j=1}^\ell f_j(n+h_j+h)\, \prod_{j=1}^\ell \overline{f_j(n+h_j)}
$$
converges to zero as $N\to\infty$. Since by our assumption $\lE_{n\in  \N}\, w(n+h)\, \overline{w(n)}=0$ for every $h\in\N$,  the result follows.
\end{proof}
\begin{proof}[Proof of Theorems~\ref{T:EDP1} and \ref{T:mf1}]
	Theorem~\ref{T:EDP1} follows immediately from Propositions~\ref{P:corrzero2} and \ref{P:zero'} (for $\ell=1$, $h_1=0$).
	
	To prove Theorem~\ref{T:mf1}, we note first that  by Theorem~\ref{T:Disjoint'},  the collection of sequences    $\{f_1,\ldots, f_\ell\}$ and $\{w\}$ are statistically disjoint. Hence, Proposition~\ref{P:Disjoint} gives
	that the difference
	$$
	\lE_{n\in[N]} |w(n)|^2\, \prod_{j=1}^\ell |f_j(n+h_j)|^2 - \lE_{n\in[N]} |w(n)|^2  \cdot \lE_{n\in[N]} \prod_{j=1}^\ell |f_j(n+h_j)|^2.
	$$
	converges to $0$ as $N\to \infty$. Using this and our assumption that the sequences $(w(n))_{n\in\N}$ and $(\prod_{j=1}^\ell f_j(n+h_j))_{n\in\N}$
	are non-null, we deduce that their product is also non-null.
	With this in mind,  Theorem~\ref{T:mf1}  follows   from  Propositions~\ref{P:corrzero} and \ref{P:zero'}.	
\end{proof}

\subsection{Proof of Corollaries~\ref{C:EDP1} and \ref{C:mf1}}
We will   need the following fact:
\begin{proposition}\label{P:ex}
Let  $P\in \R[t]$ be a non-constant polynomial with  irrational leading coefficient  
and let $\phi\colon \T\to \mathbb{U}$ be Riemann integrable. Then  the sequence  $(\phi(P(n)))_{n\in\N}$ has zero entropy,  is totally ergodic, and has a unique Furstenberg system.
\end{proposition}
\begin{remark}
In order to have total ergodicity it is essential that  the leading coefficient of $P$ (and not just any non-constant coefficient)  is irrational. For example, if  $a(n):=e(\frac{n^3}{3} +n^2\alpha$), $n\in\N$,  where $\alpha$ is irrational, then it turns out that  the sequence $(a(n))$ is not totally ergodic. We thank S.~Pattison for pointing this out, see  
Sections 5.3 and 5.4 in \cite{Pa} for  a related discussion. 
	\end{remark}
\begin{proof} Let $d:=\deg{P}$.
We start with the   well known fact (see~\cite[Section~1.7]{Fu} or \cite[Section~4.4]{Pa}) that there
exists a unipotent affine transformation  $S\colon \T^d\to \T^d$, with unique invariant measure
the Haar measure $m_{\T^d}$, so that the system $(\T^d, m_{\T^d}, S)$
is  totally ergodic (here we used that the leading coefficient of $P$ is irrational),  a Riemann integrable function $\Psi\colon \T^d\to \mathbb{U}$, and $y_0\in \T^d$, such that
 \begin{equation}\label{E:Phi}
 \Psi(S^ny_0)=\phi(P(n))\quad \text{ for every } n\in\Z.
 \end{equation}
  (For instance, when $P(n)=n^2\alpha$, $n\in\N$,  we can take
$S(t,s)=(t+\alpha,s+2t+\alpha)$,  $\Psi(t,s)=\phi(t)$, $t,s\in\T$, and $y_0=(0,0)$.)
We let $X:=\mathbb{U}^\Z$ and  $T$ be the shift transformation  on $X$. We define  the map
$\pi\colon \T^d\to X$   by
\begin{equation}\label{E:pi}
\pi(y):=(\Psi((S^ny))_{n\in\Z},\quad  \text{for } y\in\T^d.
\end{equation}
 Clearly we have $\pi\circ T=S\circ \pi$. Next,
let  $m\in \N$ and  $\ell_{-m},\ldots, \ell_m\in\Z$.
We define the function
$$
F(x):=\prod_{j=-m}^m x(j)^{\ell_j} \quad \text{ for } \ x=(x(n))_{n\in\Z}\in X,
$$
where we used the following conventions: for $z\in \mathbb{U}$ and $k<0$ we have
$z^k:=\overline{z^{-k}}$ and  $0^0=0$.
Note that the linear span of all such functions forms a conjugation closed subalgebra of $C(X)$ that separates points, hence it is  dense in $C(X)$.

Next note that for $x_0:=(\phi(P(n)))_{n\in\Z}\in X$ we have
\begin{multline*}
	\lim_{N\to\infty}\frac{1}{N} \sum_{n=1}^N F(T^nx_0)=\lim_{N\to\infty}\frac{1}{N} \sum_{n=1}^N\prod_{j=-m}^m \phi(P(n+j))^{\ell_j}\\
	=
	\lim_{N\to\infty}\frac{1}{N} \sum_{n=1}^N\prod_{j=-m}^m \Psi(S^{n+j}y_0)^{\ell_j}
	=\int_{\T^d}\prod_{j=-m}^m\Psi(S^jy)^{\ell_j}\, dm_{\T^d}(y)
	=\int_{\T^d}F\circ\pi\,dm_{\T^d},
\end{multline*}
where to justify the second identity we use \eqref{E:Phi},  for the third  we  use the unique ergodicity of $S$ and the fact that $\Psi\circ S^n$ is Riemann integrable for $n\in\Z$, and for the fourth we use \eqref{E:pi}.
By linearity and density, it follows that  the sequence of measures $(\E_{n\in[N]}\delta_{T^nx_0})_{N\in\N}$ (and hence the sequence
$(\lE_{n\in[N]}\delta_{T^nx_0})_{N\in\N}$)
converges weak-star to a measure $\mu$ on $X$, which is equal to the image of the measure $m_{\T^d}$ under $\pi$.
From the above, we deduce that   the sequence  $(\phi(P(n)))_{n\in\Z}$ has  a unique Furstenberg system, which
is $(X,\mu,T)$, and $\pi$ is a factor map from the system $(\T^d,m_{\T^d},S)$ to the system $(X,\mu,T)$.
Since the system $(\T^d,m_{\T^d},S)$ is totally ergodic and has zero entropy, the same holds for
its factor  $(X,\mu,T)$. This completes the proof.
\end{proof}

\begin{proof}[Proof of Corollaries~\ref{C:EDP1} and \ref{C:mf1}]
  It suffices to verify that  the sequence $w(n):= \phi(P(n))$, $n\in\N$,  satisfies the assumptions of Theorem~\ref{T:EDP1}. Since $P$ has a non-constant coefficient irrational, the sequence $(P(n))_{n\in\N}$ is equidistributed in $\T$, which gives that
  $  \lE_{n\in\N} |w(n)|^2=\int |\phi|^2>0$, so $w$ is non-null.
  Moreover,  it follows from Proposition~\ref{P:ex} that $w$  has zero entropy and  is totally ergodic. It remains to verify that it has vanishing self-correlations, meaning,
$$
\lE_{n\in\N}\, w(n+h)\, \overline{w(n)}=0
$$
 for every $h\in \N$.
In fact, we establish a stronger property: If $\phi,\psi\colon \T\to \C$ are Riemann integrable, then for every $h\in\N$ we have
\begin{equation}\label{E:autcor}
\E_{n\in\N}\, \phi(P(n+h))\, \psi(P(n))=\int \phi \, dm_\T\cdot \int \psi \, dm_\T.
\end{equation}
Using standard Weyl estimates this is easily shown to be the case when $\phi(t):=e(kt)$ and $\psi:=e(lt)$ for some  $k, l\in \Z$ (this is the only point where we use the assumption that $P$ has a non-linear coefficient irrational). Using linearity and uniform approximation by trigonometric polynomials,
we deduce that \eqref{E:autcor} holds for all $\phi,\psi\in C(\T)$. Finally, we
deduce that \eqref{E:autcor}
holds for all Riemann integrable $\phi,\psi$ by approximating them in $L^1(m_\T)$ by continuous functions and using that the sequence $(P(n+h))_{n\in\N}$ is equidistributed in $\T$ for every $h\in \Z$. This completes the proof.
\end{proof}

\section{Proof of main results for random weights}
\subsection{Proof of Theorems~\ref{T:Random1} and \ref{T:Random2}}
For $N\in \N$, we denote by  ${\bf 1}_{[N]}$ the indicator function of the set $[N]$ and let
$$
\mathcal{M}_N:=\{f\cdot {\bf 1}_{[N]} \  \text{ where } f\colon\N\to \mathbb{U} \text{ is multiplicative}\}.
$$
 We also  let
$B_\varepsilon$  be an $\varepsilon$-net of points in $\mathbb{U}$ of minimal cardinality (thus $|B_\varepsilon|\leq 4\varepsilon^{-2}$)
and define
$$
\mathcal{M}_{\varepsilon,N}:=\{g\in \mathcal{M}_N\colon g(k) \in B_\varepsilon \text{ for all prime powers } k\in[N] \}.
$$
We need two lemmas. The first is an approximation property.
\begin{lemma}\label{L:close}
 Let $f\colon \N\to \mathbb{U}$ be a multiplicative function. Then for every $\varepsilon>0$ and $N\in\N$, there exists $g\in \mathcal{M}_{\varepsilon,N}$ such that
 $$
 \norm{f-g}_{L^\infty[N]}\leq 2   \varepsilon  \log{N}.
 $$
 \end{lemma}
 \begin{proof}
Since $B_\varepsilon$ is an $\varepsilon$-net of $\mathbb{U}$,  and  an element of $\CM$ can take  arbitrary prescribed values on prime powers, as long as these values are taken in $\mathbb{U}$,   there exists $g\in \mathcal{M}_{\varepsilon,N}$ such that $g(1)=f(1)$ and
\begin{equation}\label{E:est}
|f(k)-g(k)|\leq \varepsilon\,
\text{ for all prime powers } \, k\in [N].
\end{equation}
 For $n\in \{2,\ldots, N\}$, let
$n=k_1\cdots k_l$, where $l\leq \log_2{N}$, be the unique factorization of $n$ into prime powers $k_1,\ldots,k_l$. Using the multiplicativity of $f$ and $g$, the estimate \eqref{E:est},  and telescoping, we get
$$
|f(n)-g(n)|=\Big|\prod_{j=1}^lf(k_j)-\prod_{j=1}^lg(k_j)\Big|\leq \varepsilon l\leq 2 \varepsilon \log{N}.
$$
 This completes the proof.
 \end{proof}
For $\varepsilon>0$ and $\ell, N\in \N$, we let
\begin{equation}\label{E:CM}
	\mathcal{M}_{\ell,\varepsilon,N}=\{(g_1,\ldots, g_\ell)\colon g_1,\ldots,g_\ell\in
	\mathcal{M}_{\varepsilon,N}\}.
\end{equation}
 The next lemma gives an  upper bound on the elements of $\mathcal{M}_{\ell,\varepsilon, N}$ that suffices for our purposes.
\begin{lemma}\label{L:count}
Let $\varepsilon>0$ and  $\ell\in \N$. Then for all large enough $N\in\N$ we have
$$
|\mathcal{M}_{\ell,\varepsilon,N}|\leq e^{4\ell \log(2\varepsilon^{-1}) \, \frac{N}{\log{N}}}
$$
\end{lemma}
 \begin{proof}
 Notice first that because of multiplicativity,
 	an $\ell$-tuple $(f_1,\ldots, f_\ell)\in \mathcal{M}_{\ell,\varepsilon,N}$
 	is uniquely determined by the values  $(f_1(k_1),\ldots, f_\ell(k_\ell))$, where $k_1,\ldots, k_\ell$ range over all prime powers in $[N]$.
 	Since for large enough $N$ there are at most $2\frac{N}{\log{N}}$ prime powers up to $N$ and  $f_j(k)\in B_\varepsilon$ for
 	$j=1,\ldots, \ell$,
 	we deduce  that
 	$$
 	|\mathcal{M}_{\ell,\varepsilon,N}|\leq  	
|B_\varepsilon|^{2 \ell \frac{N}{\log{N}}}.
$$
The asserted bound follows since $|B_\varepsilon|\leq 4\varepsilon^{-2}$.
 	\end{proof}
 Combining the previous two lemmas we can prove the following   result, which is an essential ingredient of the proofs of Theorems~\ref{T:Random1} and \ref{T:Random2}.
\begin{theorem}\label{T:Random3}
	Let $(X_n(\omega))_{n\in\N}$ be a sequence of independent random variables with $\P(X_n=-1)=\P(X_n=1)=\frac{1}{2}$, $n\in\N$. Then for every $a\colon \N\to \mathbb{U}$ we have that
	$\omega$-almost surely the following holds:
	For every  $\ell\in \N$, all   multiplicative functions $f_1,\ldots, f_\ell \colon \N\to \mathbb{U}$, and all $h_1,\ldots, h_\ell \in\Z^+$, we have
\begin{equation}\label{E:fj}
	\E_{n\in\N} \, a(n)\,   X_n(\omega)\,  \prod_{j=1}^\ell f_j(n+h_j)=0.
	\end{equation}
\end{theorem}
\begin{remarks}
$\bullet$ As was the case with Theorem~\ref{T:Random2}, the important point in  this statement is that the set of $\omega$'s for which \eqref{E:fj} holds can be chosen independently of the (uncountably many) multiplicative functions $f_1,\ldots, f_\ell\colon \N\to \mathbb{U}$.

$\bullet$ We note that for $\ell=1$  the previous result can also  be  proved using an orthogonality criterion 
by utilizing the fact that for every   $b\colon \N\to\mathbb{U}$ we have  $\omega$-almost surely
 $\E_{n\in\N} \, b(n)\,  X_{np}(\omega)\,   X_{nq}(\omega)=0$ for all $p\neq q$. But this method does not seem to be  of much help when $\ell\geq 2$
and it is the $\ell=2$ case that is needed in the proof of Theorem~\ref{T:Random1}.
\end{remarks}
\begin{proof}
Since $\ell$ and $h_1,\ldots, h_\ell$   take values on a countable set, it suffices to show that for all fixed $\ell\in \N$, $h_1,\ldots, h_\ell\in \Z^+$,   and $a\colon \N\to\mathbb{U}$,  the following statement  holds $\omega$-almost surely: For
all   multiplicative functions $f_1,\ldots, f_\ell \colon \N\to \mathbb{U}$ we have
	$$
	\E_{n\in\N} \, a(n)\,  X_n(\omega)\,  \prod_{j=1}^\ell f_j(n+h_j)=0.
	$$
To prove this, we first  note that using standard concentration of measure estimates (for example Bernstein's exponential inequality) we have for every fixed sequence $b\colon \N\to \mathbb{U}$ and every $N\in\N$ and $\delta>0$ that
\begin{equation}\label{E:Bernstein}
\mathbb{P}(|\E_{n\in [N]} X_n(\omega) \, b(n)|\geq \delta)\leq e^{-\frac{1}{4}\delta^2N}.
\end{equation}
We let
$$
\delta_N:=(\log{N})^{-1/3} \text{ and  }
\varepsilon_N:=(\log{N})^{-2}, \quad N\in \N.
$$
Using  the notation introduced in \eqref{E:CM}, we get
 for every large enough $N\in \N$ that
\begin{multline*}
\mathbb{P}\big(\sup_{(g_1,\ldots,g_\ell)\in \mathcal{M}_{\ell, \varepsilon_N, N}}|\E_{n\in [N]} \, a(n)\,  X_n(\omega) \, \prod_{j=1}^\ell g_j(n+h_j)|\geq \delta_N\big)\leq e^{-\frac{1}{4}\delta_N^2N}\
 |\mathcal{M}_{\ell,\varepsilon_N,N}|\\
 \leq
 e^{-\frac{1}{4}\delta_N^2N+4\ell \log(2\varepsilon_N^{-1}) \, \frac{N}{\log{N}}}\leq
e^{-\frac{1}{5}\frac{N}{(\log{N})^{2/3}}},
\end{multline*}
where  the first estimate follows from the union bound and \eqref{E:Bernstein}, and the second estimate follows from Lemma~\ref{L:count}.
 Using the Borel-Cantelli lemma we deduce that $\omega$-almost surely we have
 $$
\lim_{N\to\infty}\sup_{(g_1,\ldots,g_\ell)\in \mathcal{M}_{\ell, \varepsilon_N, N}}|\E_{n\in [N]}\,  a(n)\, X_n(\omega) \, \prod_{j=1}^\ell g_j(n+h_j)|=0.
$$
Using Lemma~\ref{L:close}, the fact that $\varepsilon_N\log{N}\to 0$, and telescoping, we deduce that  $\omega$-almost surely we have
$$
\lim_{N\to\infty}\sup_{f_1,\ldots,f_\ell\in  \mathcal{M}}|\E_{n\in [N]}\, a(n)\, X_n(\omega) \, \prod_{j=1}^\ell f_j(n+h_j)|=0.
$$
This completes the proof.
\end{proof}

	\begin{proof}[Proof of Theorems~\ref{T:Random1} and \ref{T:Random2}]
Let $f_1,\ldots, f_\ell$ and $h_1,\ldots, h_\ell$ be as in Theorem~\ref{T:Random2}. Note
that  $\omega$-almost surely  the sequence  $(X_k(\omega)\prod_{j=1}^\ell f_j(k+h_j))_{k\in\N}$ is   non-null, since  $\omega$-almost surely $|X_k(\omega)|=1$, $k\in\N$,  and by assumption $(\prod_{j=1}^\ell f_j(k+h_j))_{k\in\N}$ is non-null.
Likewise, if  $a\colon \N\to \mathbb{U}$ is a  non-null sequence and $f\colon \N\to \mathbb{S}$ is a multiplicative function, then  $\omega$-almost surely $(a(k) \, X_k(\omega) f(k))_{k\in\N}$ is  non-null.

Since all fixed parameters  that appear below take values on a countable set, by
 Proposition~\ref{P:corrzero2} (for Theorem~\ref{T:Random1}) and
Proposition~\ref{P:corrzero} (for Theorem~\ref{T:Random2}) it suffices to show that for every fixed $b\colon \N\to \mathbb{S}$, all
$h,\ell \in \N$, and all $h_1,\ldots, h_\ell \in\Z^+$,
	we have  $\omega$-almost surely the following  (for Theorem~\ref{T:Random1} we only need to use the case $\ell=1$, $h_1=0$):
	For all   multiplicative functions $f_1,\ldots, f_\ell \colon \N\to \mathbb{U}$ we have
	\begin{equation}\label{E:wanted}
	\E_{n\in\N} \, b(n)\, X_{n+h}(\omega)\cdot X_n(\omega)  \prod_{j=1}^\ell   f_j(n+h+h_j)\prod_{j=1}^\ell \overline{f_j(n+h_j)}=0.
	\end{equation}
(Note that then \eqref{E:wanted} also holds with $\lE_{n\in\N}$ in place of $\E_{n\in\N}$.)
We partition the positive integers into the following two sets
$$
S_1:=\bigcup_{k\in \Z^+}[2kh,(2k+1)h), \quad
S_2:=\bigcup_{k\in \Z^+}[(2k+1)h,(2k+2)h).
$$
We let
$$
Y_n(\omega):=X_{n+h}(\omega)\cdot X_n(\omega), \quad n\in \N.
$$
Note that  $\P(Y_n=-1)=\P(Y_n=1)=\frac{1}{2}$ for all  $n\in\N$. Moreover,
 for $n\in S_1$ (and fixed $h\in\N$) the random variables $Y_n(\omega)$
are independent,  and the same holds for the random variables $Y_n(\omega)$ for $n\in S_2$. For $i=1,2$ we consider independent random variables $Z_{n,i}(\omega)$, $n\in\N$, such that $\P(Z_{n,i}=-1)=\P(Z_{n,i}=1)=\frac{1}{2}$, $n\in\N$, and $Z_{n,i}:=Y_n$ for $n\in S_i$.
For $i=1,2$,  we apply  Theorem~\ref{T:Random3} for the random variables $(Z_{n,i}(\omega))_{n\in\N}$ and $a_i(n):=b(n)\, {\bf 1}_{S_i}(n)$ (then $a_i(n)\, Z_{n,i}=b(n)\, {\bf 1}_{S_i}(n)\,    Y_n$, $n\in\N$), and  deduce  that $\omega$-almost surely we have
$$
	\E_{n\in\N}\, {\bf 1}_{S_i}(n) \, b(n) \,  Y_n(\omega) \prod_{j=1}^\ell f_j(n+h+h_j)\prod_{j=1}^\ell \overline{f_j(n+h_j)}=0
	$$
for $i=1,2$. Adding the two identities we get \eqref{E:wanted}. This completes the proof.
		\end{proof}

\subsection{Proof of Theorem~\ref{T:Random1'}}\label{SS:1.6}
We will use the following finitistic strengthening of Theorem~\ref{T:Tao} that  can be deduced from Theorem~\ref{T:Tao}  using a compactness argument:
\begin{theorem}\label{T:Taofinite}
For every $C>0$ there exists $m\in \N$ such that for every sequence $a\colon [m]\to \S$ there exist $d,n\in \N$ with $dn\leq m$ such that
$|\sum_{k=1}^na(dk)|>C$.
\end{theorem}
We deduce from this some necessary conditions for a sequence to be a good weight for the Erd\H{o}s discrepancy problem.
\begin{lemma}\label{L:goodEDP}
Let $w\colon \N\to \C$ be a sequence and $c\in \C\setminus\{0\}$. Suppose  that for infinitely many $m\in \N$ there exists
 $r\in \N$ such that
 \begin{equation}\label{E:wrm}
w\big(r\frac{m!}{i}+j\big)=c \quad \text{ for all } i,j\in\{1,\ldots m\}.
\end{equation}
Then $w$ is a good weight for the Erd\H{o}s discrepancy problem.
\end{lemma}
\begin{remark}
The conclusion fails if we simply assume that $w$  is equal to  a non-zero constant on  a union of arbitrarily long intervals. To see this, let $(a(k))_{k\in\N}$ be a completely multiplicative function that is equal to  $(-1)^n$  on a sequence of  intervals with  lengths even numbers that increase to infinity (such a multiplicative function can be explicitly constructed). Let also $w$  be the indicator function of the union of this sequence of intervals. Then
$\sup_{d, n\in \N}\big|\sum_{k=1}^n a(dk) \, w(k)\big|\leq 1$.
\end{remark}
\begin{proof}
Let  $a\colon \N\to \S$ be a sequence and $C>0$.
Let $m\in\N$ be  so that   Theorem~\ref{T:Taofinite} applies for this $C$ and \eqref{E:wrm} holds for some $c\in \C\setminus\{0\}$ and  $r\in \N$.
 We use   Theorem~\ref{T:Taofinite}
for the sequence $(a(rm!+k))_{k\in[m]}$ and we get that there exist $d,n\in \N$, with $dn\leq m$,  such that
\begin{equation}\label{E:rmdk}
\Big|\sum_{k=1}^na(rm!+dk)\Big|\geq \frac{C}{|c|} .
\end{equation}
We let
$$
S_d(N):=\sum_{k=1}^Na(kd)\, w(k), \quad N\in \N.
$$
Note that
$$
S_d\big(r\frac{m!}{d}+n\big)-S_d\big(r\frac{m!}{d}\big)=\sum_{k=1}^na(rm!+dk)\,  w\big(r\frac{m!}{d}+k\big).
$$
Since $d, n\leq m$, using the previous identity, and  \eqref{E:wrm}, \eqref{E:rmdk}, we deduce that
$$
\big|S_d\big(r\frac{m!}{d}+n\big)-S_d\big(r\frac{m!}{d}\big)\big|= |c| \Big|\sum_{k=1}^na(rm!+dk)\Big|\geq  C.
$$
Hence,  either $|S_d\big(r\frac{m!}{d}+n\big)|\geq \frac{C}{2}$ or $|S_d\big(r\frac{m!}{d}\big)|\geq \frac{C}{2}$. Since $C$ was arbitrary, we deduce that
$\sup_{d,N\in\N}|S_d(N)|=+\infty$. This  completes the proof.
\end{proof}
\begin{proof}[Proof of Theorem~\ref{T:Random1'}]
Let $c\in \C\setminus\{0\}$ be such that
$\sum_{k\in \N}\rho_k^l=+\infty$ for every $l\in \N$,  where $\rho_k:=\P(X_k=c)$, $k\in \N$.
Let $m\geq 4$.
By Lemma~\ref{L:goodEDP}, it suffices to show  that $\omega$-almost surely  there exists
 $r\in \N$ such that
 $$
X_{r\frac{m!}{i}+j}(\omega)=c \quad \text{ for all } i,j\in [m].
$$
  One easily verifies that for any fixed $m\geq 4$ the random variables
$X_{r\frac{m!}{i}+j}$, $i,j\in[m]$,  $r\in m!\N+1$, are independent. Hence,
$$
\P\big(X_{r\frac{m!}{i}+j}(\omega)=c \  \text{ for all } i,j\in [m]\big)=
\prod_{i,j\in[m]}\P\big(X_{r\frac{m!}{i}+j}(\omega)=
c \big)
=\prod_{i,j\in[m]}\rho_{r\frac{m!}{i}+j}.
$$
Since $(\rho_k)_{k\in \N}$ is decreasing, we have  that $\prod_{i,j\in[m]}\rho_{r\frac{m!}{i}+j}
\geq \rho^{m^2}_{r(m+1)!}$ for all $r\in \N$. Moreover, since $ \sum_{k\in \N} \rho_k^{m^2}=+\infty$,
using again that  $(\rho_k)_{k\in \N}$ is decreasing,
we get that $\sum_{r\in m!\N+1} \rho^{m^2}_{r(m+1)!}=+\infty$. We deduce that
$$
\sum_{r\in m!\N+1}\P\big(X_{r\frac{m!}{i}+j}(\omega)=c  \text{ for all } i,j\in [m]\big)=+\infty.
$$
Since the  sets involved in the above  probabilities are independent,  the Borel-Cantelli theorem applies, and gives
that  $\omega$-almost surely for infinitely many $r\in \N$ we have that
$X_{r\frac{m!}{i}+j}(\omega)=c$  for all $i,j\in [m]$.
This completes the proof.
\end{proof}

\section*{Acknowledgments} 
  I would like to thank    M.~Kolountzakis for
providing the proof of Theorem~\ref{T:Random1'} and other useful remarks, 
and S.~Pattison
 for pointing out a correction in the statement of Proposition~\ref{P:ex}.
I would also like to thank the American Institute of Mathematics (AIM) for its hospitality;   part of this work was motivated by problems raised during the 2018 workshop ``Sarnak's Conjecture''.

\bibliographystyle{amsplain}



\begin{dajauthors}
	\begin{authorinfo}[Nikos]
		Nikos Frantzikinakis\\
		Department of Mathematics and Applied Mathematics\\
		University of Crete\\
		Voutes University Campus,
		Heraklion 70013, Greece\\
		frantzikinakis \imageat{} gmail\imagedot{}com  \\
		\url{http://users.math.uoc.gr/~frantzikinakis/}
	\end{authorinfo}
\end{dajauthors}

\end{document}